\documentclass[12pt,leqno,twoside]{amsart}
\makeatletter
\@namedef{subjclassname@2020}{\textup{2020} Mathematics Subject Classification}
\makeatother

\usepackage[latin1]{inputenc}
\usepackage[T1]{fontenc}
\usepackage[colorlinks=true, pdfstartview=FitV, linkcolor=blue, citecolor=blue, urlcolor=blue]{hyperref}
\usepackage{amstext,amsmath,amscd, bezier,indentfirst,amsthm,amsgen, geometry}
\usepackage[all,knot,arc,import,poly]{xy}
\usepackage{amsfonts,color, soul}  %\st{} for overstriking
\usepackage{amssymb}
\usepackage{latexsym}
\usepackage{epsfig}
\usepackage{graphicx}
\usepackage{srcltx}
\usepackage{enumerate}
\UseRawInputEncoding

\newtheorem{theorem}{Theorem}[section]
\newtheorem{corollary}[theorem]{Corollary}

\newtheorem{lemma}[theorem]{Lemma}
\newtheorem{proposition}[theorem]{Proposition}
\theoremstyle{definition}
\newtheorem{definition}[theorem]{Definition}
\theoremstyle{remark}
\newtheorem{remark}[theorem]{\sc Remark}

\newtheorem{example}[theorem]{\sc Example}
\numberwithin{equation}{section}

\newcommand{\Sing}{{\rm{Sing\hspace{1pt}}}}

\newcommand{\Disc}{{\rm{Disc\hspace{1pt}}}}

\newcommand{\im}{{\rm{Im}}}

\newcommand{\grad}{\mathop{\rm{grad}}}
\newcommand{\gradx}{\mathop{\rm{grad}_{x}}}

\newcommand{\e}{\varepsilon}
\newcommand{\m}{\setminus}
\newcommand{\ord}{{\rm{ord}}}
\newcommand{\fin}{\hspace*{\fill}$\Box$\vspace*{2mm}}

\newcommand{\cQ}{{\mathcal Q}}

\newcommand{\cS}{{\mathcal S}}

\newcommand{\cW}{{\mathcal W}}

\newcommand{\bK}{{\mathbb K}}
\newcommand{\bR}{{\mathbb R}}
\newcommand{\bC}{{\mathbb C}}
\newcommand{\bP}{{\mathbb P}}
\newcommand{\bN}{{\mathbb N}}

\begin{document}
\title[Deformations with fibre constancy]{Deformations with fibre constancy}

\author{\sc Ying Chen}
\address{School of Mathematics and Statistics, HuaZhong University of Science and Technology WuHan 430074, P. R. China}
\email{ychenmaths@hust.edu.cn}

\author{\sc Cezar Joi\c{t}a}
\address{Institute of Mathematics of the Romanian Academy, P.O. Box 1-764,
 014700 Bucharest, Romania.} % and Laboratoire Europ\' een Associ\'e  CNRS Franco-Roumain Math-Mode}
\email{Cezar.Joita@imar.ro}

\author{Mihai Tib\u{a}r}
\address{Univ. Lille, CNRS, UMR 8524 -- Laboratoire Paul Painlev\'e, F-59000 Lille,
France}
\email{mihai-marius.tibar@univ-lille.fr}


% 32S30 Deformations of complex singularities; vanishing cycles
% 14B07 Deformations of singularities
% 32S55 Milnor fibration; relations with knot theory
% 14D06 Fibrations, degenerations in algebraic geometry 

% 32S15 Equisingularity (topological and analytic) 

% 32C18 Topology of analytic spaces
% 14B05 Singularities in algebraic geometry
% 58K05 Critical points of functions and mappings on manifolds 
% 57R45 Singularities of differentiable mappings in differential topology
% 14P10 Semialgebraic sets and related spaces
% 14P15 Real-analytic and semi-analytic sets 
% 32S20 Global theory of complex singularities; cohomological properties 
% 32S60 Stratifications; constructible sheaves; intersection cohomology


\begin{abstract}
For studying the local topology of maps, one uses deformations which split the singularities into simpler ones while preserving the general fibres.
We give conditions under which such conservation holds.
  \end{abstract}

% This extends to maps the study \cite{JST} for the case of function germs.

\maketitle

\section{Introduction}

% , and its associated map $\widetilde F :(\bK^n\times\bK, 0)\to(\bK^m \times \bK, 0)$, 
% $\widetilde F(x,t) = \bigl( F(x,t), t\bigr)$.  
%We have $\Sing \widetilde F = \{ (x,t) \mid  x\in \Sing F_{t}\}\subset \bK^n\times\bK$ as set germ at $0$.

%Saying that a deformation of a map germ is topologically trivial ``at the boundary'' means that we are interested in deformations which preserve the topology of the Milnor fibre. 

Our aim is to characterise deformations of analytic maps $F_{t}(x)$ of parameter $t\in \bR$ close to 0 for which the general fibres of $F_{0}$ are preserved as global fibres within a fixed  ball neighbourhood of the origin. 
  This property will be called here \emph{fibre constancy},  see more precisely Definition \ref{d:topconstantaway} below. 
  
  Such deformations allows one to capture the topology of the general fibres of  $F_{0}$ by splitting the singular locus of $F_{0}$ into simpler singularities of $F_{t}$. 
This method has been brought to light by Brieskorn \cite{Br} in the case of holomorphic functions with isolated singularities, see also \cite{JiT}, and more interestingly,  by Siersma and his school in the case of holomorphic functions with non-isolated singularities under the name of ``admissible deformations'', see e.g. \cite{Si},  \cite{Sc}, \cite{dJ}, \cite{Pe}, \cite{Za},  \cite{Fe}, and some more recent papers such as \cite{FM}, \cite{ST-mildefo}, \cite{MPT}.
%In particular such deformations cannot be topologically equisingular families of map germs. 
Two very recent papers have been dedicated to establishing various conditions which insure the ``fibre constancy''    for complex functions \cite{Ho}, and for both real and complex functions \cite{JST}. 

%If the fibre constancy is insured, then the splitting of a complicated singularity into 
%simpler ones allows the reconstruction of the Milnor fibre of $F_{0}$ out of the Milnor fibres of the 
%simpler singularities of $F_{t}$.  
%We refer to the next page for a new application. 

\medskip

We address here the general case of  map germs, especially $m\ge2$, in both real and complex settings, and we find several new conditions insuring the fibre constancy. 
 
To set the notations, let
$F:(\bK^n\times\bK, 0)\to(\bK^m, 0)$, $n\ge m\ge 1$, $\bK = \bR$ or $\bC$,  be a $\bK$-analytic map germ  regarded as a one-parameter deformation $F_{t}(x)$ of the map germ $F_{0}(x):=F(x ,0)$. We consider the associated map  germ $\widetilde F = (\bK^n\times\bK, 0)\to(\bK^m\times \bK, 0)$,  $\widetilde F(x,t) := \bigl( F(x,t), t\bigr)$. 

For the simplicity, we fix $\bK = \bR$  in the next definition.  Let $\rho$ denote the square of the Euclidean distance function in $\bR^{n+1}$. 
%We first state a key definition in our setting $m>1$ (which extends the one used in \cite{JST} for $m=1$):
%%%%%
\begin{definition}[Tame deformations]\label{d:tamedefo}\ \\
 We say that $F_{t}(x) = F(x,t)$  is a \emph{tame deformation} of $F_{0}$ if the following condition holds
 (called \emph{$\rho$-regularity}, cf Definition \ref{d:tamemap}):
\begin{equation}\label{eq:rhoreg}
M(\widetilde F)  \ \bigcap F_{0}^{-1}(0) \cap \Sing F_{0} \subset \{(0,0)\},
\end{equation}
where $M(\widetilde F)   := \overline{\Sing (\widetilde F,\rho)\m \Sing \widetilde F}$.
\end{definition}

Tame deformations insure that the image of $\widetilde F$ is well-defined as a set germ at $(0, 0)\in \bR^{m}\times \bR$, and that
 the image $\widetilde F (\Sing \widetilde F)$ is a well-defined set germ at $(0, 0)\in \bR^{m}\times \bR$, in which case  one usually calls it \emph{discriminant} and denotes it by $\Disc \widetilde F$ (cf Proposition \ref{t:images}).  
 The  $\rho$-regularity of Definition \ref{d:tamedefo} also implies (cf. \cite[Proposition 4.2]{ART}) the key property that  $\widetilde F$ has a locally trivial fibration outside its discriminant, which one calls  \emph{Milnor-Hamm fibration}, see Definition \ref{d:Mfib}.  

With this preparation at hand, we may now state our announced definition (see also \cite[Definition 1.1]{JST}):

 %%%%%%%%%%%%%%%%%
\begin{definition}[Deformations with fibre constancy]\label{d:topconstantaway} \ \\ %{d:tube}
 Let $F_{0}:(\bR^n, 0)\to(\bR^m, 0)$, $n\ge m\ge 2$ be a non-constant analytic map germ.  We say that the deformation $F$ of $F_{0}$ is a \emph{deformation with fibre constancy} if $\widetilde F$ has a Milnor-Hamm fibration. 
\end{definition}

 This tells  the following: if  $B_{r}\subset \bR^{n}\times \bR$ and by $B_{\delta}\subset \bR^{m}\times \bR$ are those balls of radii $r>0$ and $\delta>0$,  centred at the respective origins, which occur in Definition \ref{d:Mfib} for $G:=\widetilde F$, then we have the diffeomorphism of fibres:
  $$B_{r} \cap  F_{0}^{-1}(a) \simeq B_{r} \cap \widetilde F^{-1}(\lambda_{a})$$ 
  for any    %% however the complement of Disc \widetilde F may have more connected components than 
  %%% the complement of Disc F_{0}
 $a\in (B_{\delta} \cap  \{t=0\}) \m \Disc F_{0}$ and any $\lambda_{a} \in  B_{\delta} \m \Disc \widetilde F$ which belongs to the same connected component\footnote{Let us remark the equality $\Disc F_{0} = \{t=0\} \cap \Disc \widetilde F$.} of $B_{\delta} \m \Disc \widetilde F$ as $a$. In particular, the fact that $\widetilde  F$ has  a locally trivial fibration over the complement of the discriminant $\Disc \widetilde F$ implies that $F_{0}$ has a locally trivial fibration over the complement of its own discriminant $\Disc F_{0} = \{ t=0\} \cap \Disc \widetilde F$. 
 
 %\red{Question:  if there is no equality, then what is the relation between the connected components of the complements??}

%%%%%%%%%%%%%
 The above
definition tacitly assumes that the involved maps have well defined images and discriminants.  Throughout the paper we will take care that our hypotheses insure this property too.

\ 

Let us briefly explain an application of the above defined fibre constancy.  Let  $F_{0}:(\bC^n, 0)\to (\bC^m, 0)$ define  a \emph{complete intersection} with non-isolated singular locus $\Sing F_{0} := \Sigma_{0}$ of dimension 1, and consider a deformation $F:(\bC^n\times\bC, 0)\to (\bC^m, 0)$ with fibre constancy in the sense of Definition \ref{d:topconstantaway}. Therefore  $F_{0}$ has a Milnor-Hamm fibration, and thus  the topology of its Milnor fibre can be studied by extending the technique developed in the paper \cite{ST-mildefo} for ``admissible deformations'' of a function germ. More precisely, the singular set $\Sing F_{0}$ deforms into $\Sing F_{t}$ which is a disjoint union  of a 1-dimensional singular set $\Sigma_{t}$ and a finite set $P_{t}$ of isolated singularities. Outside a certain finite set $Q_{t}\subset \Sigma_{t}$ of ``special points'',  on  any connected component of $\Sigma_{t}\m Q_{t}$, the map $F_{t}$ has the transversal type of an ICIS  (isolated complete intersection singularity), and the corresponding transversal  Milnor fibre is endowed with a \emph{Milnor monodromy} and  with a \emph{vertical monodromy}\footnote{For all this terminology, the reader is referred to \cite{ST-mildefo} and its included references.}. Then, as described in \cite{ST-mildefo}, these data may be patched together in order to build the homology of the Milnor fibre of $F_{0}$.  

\bigskip

In the general setting of real and complex map germs, this paper  focusses on producing
handy criteria for deformations with fibre constancy in terms of the \emph{partial Thom regularity} (also denoted by \emph{$\partial$-Thom regularity}, cf Definition \ref {d:partialThom}).
We first show that $\partial$-Thom regularity implies $\rho$-regularity, which in turn implies the existence of the Milnor-Hamm fibration needed in Definition \ref{d:topconstantaway}, cf \cite{ART}, \cite{JT}. 

 Section \ref{s:loja} contains our  main results.  We show how to control the key $\partial$-Thom regularity by using inequalities of \L ojasiewicz type for map germs in Theorem \ref{t:deformation}, and by using a Parusi\' nski type inequality in Theorem \ref{t:tame}, respectively. In each case, proofs are new and also radically different with respect to what had been done before in some particular contexts.
 
Section \S\ref{s:ag} treats the composition of deformations in very large generality. We discuss and provide a general answer, cf Theorem \ref{t:Thomcompo} and Example \ref{ex:Thomcompo}, to a problem which occurs in a particular setting in \cite{AG}.

 %What inequalities? see also Polish papers, or not.
 
 \medskip
\textbf{Acknowledgements.} Ying Chen acknowledges the support from the National Natural Science Foundation of China (NSFC) (Grant no. 11601168). Cezar Joi\c ta and Mihai Tib\u ar acknowledge support from GDRI ECO Math,  from the Labex CEMPI (ANR-11-LABX-0007-01),
and the support of  the project ``Singularities and Applications'' -  CF 132/31.07.2023 funded by the European Union - NextGenerationEU - through Romania's National Recovery and Resilience Plan.

\bigskip
 
 %%%%%%%%%%%%%%%
\section{Terminology and preliminary results}\label{s:terminology}

Let $G:(\bK^n,0)\rightarrow(\bK^p,0)$ be a $\bK$-analytic map germ, where $\bK$ is $\bR$ or $\bC$. We will denote by $V(G)$ the central fibre $G^{-1}(0)$.

\begin{definition}[$\rho$-regular map]\label{d:tamemap}\ \\
Let $\rho$ denote the square of the Euclidean distance function in $\bK^n$. 
Let 
$$M(G) := \overline{\Sing (G,\rho)\m \Sing G}$$
be the \emph{Milnor set of $G$}.
 We say that $G$  is a \emph{$\rho$-regular map} map germ if:
\begin{equation}\label{eq:rhoreggen}
 M(G)\cap V(G) \cap \Sing G \subset \{(0,0)\}.
\end{equation}
\end{definition}

We have recently considered the question what conditions  insure the tameness in the composition of map germs.

%%%%%%%%%%%%%%%
\begin{definition}[$\partial$-Thom regularity]\label{def:stratified thom}\label{d:partialThom} \ \\
	Let $G:(\bK^n,0)\rightarrow(\bK^p,0)$ be a $\bK$-analytic map germ. We say that $G$ is \textit{$\partial$-Thom regular} if there exists a Whitney (a)-stratification $\cW$ of some open ball $B$ centred at $0\in \bK^n$ such that $B\m G^{-1}(G(B\cap \Sing G))$ and $\{0\}$ are strata, that $B\cap V(G)$ and $B\cap V(G)\cap \Sing G$ are unions of strata,  and that  the  pair of strata 
	$$\Bigl(B\m G^{-1}\bigl(G(B\cap \Sing G)\bigr), W\Bigr)$$
satisfies the
	Thom ($a_{G}$)-regularity condition for any  stratum $W\subset B\cap V(G)\setminus\{0\}$.
\end{definition}
%%%

Comparing to \cite[Definition 5.7]{JT}, we observe that the above definition is a particular case of 
the $\partial$-Thom regularity used in \cite{JT} in the setting where the singular locus itself has a stratification and one deals with a singular fibration.

%%%%%%%%%%%%%%%%%
\subsection{The image problem for map germs}
The image of a map germ is not necessarily well-defined as a set germ. We refer to \cite{ART}, \cite{JT}, \cite{JT2}, \cite{JT3} for details, examples, and recent results.

Let us first recall the following notion: for $U,V \subset \bK^{n}$ subsets containing the origin, the set germs $(U, 0)$ and $(V, 0)$ are equal if and only if there exists some open ball $B_{\varepsilon} \subset \bK^{n}$ centred at $0$ of radius $\varepsilon > 0$ such that  $ U\cap B_{\varepsilon} = V \cap B_{\varepsilon}$.

\begin{definition}\cite[Definition 2.2]{ART}, \cite{JT}\label{d:setgerm}
Let $G:(\bK^n,0)\rightarrow(\bK^p,0)$, $n \geq p > 0$, be a continuous map germ, where $\bK$ is $\bR$ or $\bC$. We say that the image by $G$ of a set $K \subset \bK^n$ containing $0$ is a well-defined set germ at $0 \in \bK^p$ if  the set germ
$\bigl(B_{\e}\cap G(K), 0\bigr)$ is independent of the small enough radius $\e >0$.
\end{definition}

%%%

%%%%%%%%%%%%%
\begin{proposition}\label{t:images}
Let $G:(\bK^n,0)\rightarrow(\bK^p,0)$ be a $\bK$-analytic map germ which is $\partial$-Thom regular. Then 
$G$ is $\rho$-regular, and
 the images  $\im G$ and $G(\Sing G)$ are well-defined as set germs at $0\in \bK^p$, cf Definition \ref{d:setgerm}.
 \fin
\end{proposition}
\begin{proof}
The existence of a $\partial$-Thom stratification on some open ball  $B$ 
 implies that there is  $R>0$ such that,  for any positive $r\le R$, the sphere $S_{r}\subset \bK^{n}$ is transversal to all  positive dimensional strata $W\in\cW$ such that $W\subset V(G)$. It follows that
   the sphere $S_{r}$ is transversal to the smooth nearby fibres of $G$, and therefore $G$ is $\rho$-regular.  
   
The  $\rho$-regularity implies that $\im G$ and $G(\Sing G)$ are well-defined set germs, as shown  in \cite[Theorem 4.5 (a)]{JT}.
\end{proof}

\noindent \textbf{Notation}.  If  the image $G(\Sing G)$ is a well-defined set germ at the origin, then we will denote it by $\Disc G$, and usually call it  ``the discriminant of $G$''. We say that the map germ $G$ is \emph{nice} if it has well defined
image and discriminant as set germs at $0\in \bK^p$.

%%%%%%%%%%%
\begin{definition}[\emph{Milnor-Hamm fibration}]\label{d:Mfib}
Let $G: (\bR^n,0) \rightarrow (\bR^p,0)$ be a non-constant nice analytic map germ. We say that $G$ has \emph{Milnor-Hamm fibration} if, for any $\varepsilon>0$ small enough, there exists $0<\delta\ll\varepsilon$ such that the restriction:
\begin{equation}\label{eq:Mfib}
		G_{|}:B^{n}_{\varepsilon}\cap G^{-1}(B^{p}_{\delta}\m \Disc G)\rightarrow B^{p}_{\delta}\m \Disc G
\end{equation}
is a $C^{\infty}$ locally trivial fibration over each connected component of $B^{p}_{\delta}\m \Disc G$, such that it is independent of the choice of $\varepsilon$ and $\delta$, up to diffeomorphisms.
\end{definition}
 
  It has been shown in \cite[Proposition 4.2]{ART} that: \emph{if $G$ is $\rho$-regular then $G$ has a Milnor-Hamm fibration}. By  Proposition \ref{t:images}, we then get the following consequence:
  \begin{corollary}\label{c:milnorfib}
If $G:(\bK^n,0)\rightarrow(\bK^p,0)$ is a $\bK$-analytic map germ which is $\partial$-Thom regular, then
$G$  has a  Milnor-Hamm tube fibration. 
\fin
\end{corollary}

%%%%%%%%%%%%%%%%%%%

\section{$\partial$-Thom regularity of deformations}\label{s:loja}
% via \L ojasiewicz type inequalities

%%%%%%%%%%%%%%%%%%

%%%%%%%%%%%%%%%%%%

Let $\psi : (\bR^n, 0) \to (\bR^m, 0)$ be some analytic map germ, $n \ge m \ge 2$, and let $V := \psi^{-1}(0)$.

We discuss here several ways of checking the partial Thom regularity that we assume in the statement of the main theorem.

The idea is to control the growth of functions by inequalities.

%%%%%%%%%%%%%%%%%%%%%%%%%%%%%%%%%%%%%%%

\subsection{Thom regularity via the \L ojasiewicz inequality}\label{ss:classicandnew}
To prove that the Thom (a$_f$)-regularity holds for functions $f$, Hamm and L\^ e showed in \cite{HL} how to use the existence of the \L ojasiewicz inequality for $\bK$-analytic function germs $f: (\bK^{n}, 0) \to (\bK, 0)$, namely: there exists some $0< \theta < 1$, such that for any $x$ in some neighbourhood of $0$ one has
\begin{equation}\label{eq:lo}
 \| f(x) \|^{\theta} \le \| \grad f (x) \| ,
\end{equation}
where $\grad f (x)$ denotes here the conjugate of the complex gradient. 

 In the setting of real analytic map germs $\psi : (\bR^{n}, 0) \to (\bR^m, 0)$,  Hamm and L\^e's method of proof may still work when assuming a \L ojasiewicz type inequality.   Massey in \cite{Ma}  points out that if the following condition:
\begin{equation}\label{eq:lomassey}
 \| \psi (x) \|^{\theta} \le  K \nu(x),
\end{equation}
holds for some $0< \theta < 1$ and some $K>0$ in some neighbourhood of $0$, where:
\begin{equation}\label{eq:distance} 
 \nu(x) := \min_{\| a\| = 1}  \| \sum_i a_i \grad \psi_i (x)\|
\end{equation}
is the Rabier distance function, cf \cite{Ra}.  Massey shows that this implies the existence of a Thom (a$_{\psi}$)-regular stratification of  $(V, 0)$.

%%%%%%%%%%%%%%%%%
%%%%%%
\subsection{\L ojasiewicz type inequality in case of deformations of maps}\

Let  $F:(\bR^n\times\bR, 0)\to(\bR^m, 0)$ be an analytic map germ viewed as one-parameter deformation $F_{t}(x) := F(x,t) = (f_{1}(x,t),\dots,f_{m}(x,t))$, and let $\widetilde F (x,t) := (F(x,t), t)$ be the associated map germ.

Let us  set the notation:
\begin{equation}\label{eq:niu}
  \nu_{F_{t}}(x) := \min_{\| a\| = 1}  \Bigl\| \sum_{i=1}^{m} a_i \gradx f_i (x,t) \Bigr\| ,
\end{equation}
where $\gradx f_{i} := \bigl(\overline{\frac{\partial f_{i}}{\partial x_{1}}}, \ldots , \overline{\frac{\partial f_{i}}{\partial x_{n}}} \bigr)$ thus contains the partial derivatives with respect to the variables $x$ only.

 We show a condition under which $\widetilde F$ is $\partial$-Thom regular.
%%%%%%%%
\begin{theorem}\label{t:deformation}
Let $F_{t}(x) = F(x,t)$ be a deformation of the map germ $F_{0}:=F(x,0) : (\bR^n, 0)\to(\bR^m, 0)$. Assume that the following condition holds:

\begin{equation}\label{eq:cond1}
  \left\{ \begin{aligned}
\quad &  \mbox{ There exists }  0<\theta<1 \mbox{ such that for any } x_{0}\in F^{-1}(0)\cap \Sing F_{0}\m\{0\} \\ 
\quad &  \mbox{ there is a constant } c(x_{0}) >0   \mbox{ for which the following inequality holds: }\\
\quad & \|F(x,t)\|^{\theta}\leq c(x_{0}) \nu_{F_{t}}(x)   \  \mbox{ when } (x,t) \to (x_{0},0), \ (x,t) \notin \Sing \widetilde F .
   \end{aligned}\right.
\end{equation}

Then the associated map germ $\widetilde F$ is $\partial$-Thom regular. %,  and therefore it is  $\rho$-regular.
\end{theorem}

%%%%%%%%
\begin{remark}
It may happen that $\widetilde F$ is $\partial$-Thom regular without $F$ being $\partial$-Thom regular.	
  For instance, 
 let $F(x,y,t)=(f_{1}, f_{2}) := (x, y(x^{2}+y^{2})+xt^{2})$, and let 
 $\widetilde F:(\bR^3,0)\to(\bR^{3},0)$, $\widetilde F(x,y,t)=(x, y(x^{2}+y^{2})+xt^{2}, t)$ be the associated map.
   Since $V(\widetilde F)=\{(0,0,0)\}$, the map germ 
 	$\widetilde F$ is $\partial$-Thom regular by definition. According to \cite[Example 5.2]{ART}, $F(x,y,t)=(x, y(x^{2}+y^{2})+xt^{2})$ is not $\partial$-Thom regular. 
%We observe that $\Sing\widetilde F\cap V(F)=\{x=y=0\}\varsubsetneq\{t=0\}\times\bR^{2}$. 
\end{remark}

\begin{remark}
The inequality within \eqref{eq:cond1} is a Lojasiewicz type condition. In case of map germs but without reference to deformations,  Massey used such a condition in \cite{Ma} over a full neighbourhood of the origin, and with the restriction $\Sing F\subset F^{-1}(0)$. In the setting of deformations of map germs,  Massey's condition appears to be too rigid, since it does not allow deformations where $\Sing F_{0}$ splits outside the central fibre $F^{-1}(0)$ and which are precisely the object of many papers in the literature, e.g. \cite{Si}, \cite{dJ}, \cite{Pe}, \cite{Za}.

 In contrast, our theorem includes such splittings since it concerns deformations of map germs without restrictions 
on the singular locus; observe that  condition \eqref{eq:cond1} refers only non-singular points $(x,t) \notin\Sing\widetilde F$. 
Our method of proof builds on the idea used by Hamm and L\^{e} in \cite{HL} to prove the \L ojasiewicz classical inequality for analytic functions.
\end{remark}

\begin{proof}
Let $F(x,t)=(f_{1}(x,t),\dots,f_{m}(x,t))$, and let $\widetilde F(x,t):=(F(x,t),t)$. We consider the germ at the origin of the following analytic set of dimension $n+1$:
$$X=\{f_{1}(x,t)-s^{L}_{1}=\dots=f_{m}(x,t)-s^{L}_{m}=0\} \subset \bR^{n}\times \bR\times \bR^{m}$$
where $L\in\bN$ is sufficiently large such that  $0<\theta <\frac{L-1}{L}$.  Let $\cW$ be a Whitney (a)-regular stratification $\cW$ of the set germ $X$, such that $\Sing X$ is a union of strata. The stratified singular locus $\Sing_{\cW} \pi$ of the function germ $\pi:(X,0)\longrightarrow(\bR,0)$, $(x,t,s)\mapsto t$, is a closed analytic closed set, and it is included in the fibre $X\cap \{t=0\}$. As proved by Hironaka, the Whitney stratification $\cW$ may be refined into a Whitney stratification which is also Thom  $(a_{\pi})$-regular and such that $\Sing_{\cW} \pi$ is a union of strata.  This shows that there exists a $\partial$-Thom stratification $\cS$ of $\Sing_{\cW} \pi$, i.e. satisfying the following condition:  \emph{for any stratum $S\in \cS$, the pair $\bigl(X\m (\{t=0\} \cup\Sing X) , S\bigr)$ is Thom  $(a_{\pi})$-regular.}

We consider now the slices $S\cap \{s=0\}$ of all the strata $S\subset \Sing_{\cW} \pi$. Since these slices are not necessarily non-singular, one needs to refine the partition into a Whitney (a)-regular stratification $\cS'$ of $X\cap \{t=0\}\cap \{s=0\}$. Since $\cS'$ is a refinement of $\cS$, it follows that, for any stratum $W\in \cS'$, the  pair
\begin{equation}\label{eq:thom1}
\Bigl(X\m (\{t=0\} \cup\Sing X) , W\Bigr) 
\end{equation}
is Thom  $(a_{\pi})$-regular.

We have the equality $\Sing \widetilde F \cap V(\widetilde F)\times\{0\}^{m} = \Sing X \cap \Sing \pi \cap \{s=0\}$, where $\Sing \pi \subset \{ t=0\}$. 
Thus $\cS'$ is a stratification of the set $\Sing \widetilde F \cap V(\widetilde F)$, and we will show that 
it is a partial Thom stratification of $\widetilde F$, i.e. that the pair  $\bigl(B^{n+1}_{\varepsilon}\m\widetilde F^{-1}(\Disc(\widetilde F), W\bigr)$ is (a$_{\widetilde F}$)-regular for all $W\in \cS'$.

  We consider sequences of points $(x, t)\subset B^{n+1}_{\varepsilon}\m\widetilde F^{-1}(\Disc(\widetilde F))$ such that $x\rightarrow x_{0}$ and $t\rightarrow 0$ where $(x_{0},0)\in\Sing\widetilde F\cap V(\widetilde F)$, and
   such that $(x_{0},0)\in W$ for a positive dimensional stratum $W\in\cS'$.
     Note that for a corresponding triple $(x,t,s)\in X$, the variable  $s$ converges $0\in\bR^{m}$, since $F(x_{0},0)=0$.
     
          By \eqref{eq:thom1},  the inclusion:
\begin{equation}\label{eq:inclu13}
T:=\lim_{(x,t,s)\rightarrow(x_{0},0,0)}T_{(x,t,s)}\bigl(\pi^{-1}(t)\cap X\bigr)\supset T_{(x_{0},0,0)}W
\end{equation}
holds, where $\dim T_{(x_{0},0,0)}W\geq 1$, and by  tacitly assuming that the limit $T$ exists in the appropriate Grassmannian. Note that we have  $\dim T=n$.

\begin{lemma}\label{l:keylemma}
Under the hypotheses of Theorem \ref{t:deformation}, one has the equality:
 $$T = A \times \{0\}_{1} \times \bR^{m}\subset  \bR^{n}\times \bR \times \bR^{m},$$
 where $A \subset \bR^{n}$ is some linear subspace of dimension $n-m$.
 % defined as the dual of  the space 
 %$\span {\bigl\langle \lim_{x\to x_{0}, t\to 0}\grad_{x} f_{i}(x,t)\bigr\rangle}_{i=1,\ldots , m}$.
\end{lemma}
%\red{By definition $T$ is the limit of the span$^{\perp}$. What is computed in the proof is 
%the span$^{\perp}$ of the limits. In general, these are two different things. See the counterexample.}
\begin{proof}
The normal space to  $T_{(x,t,s)}\bigl(\pi^{-1}(t)\cap X\bigr)$
is spanned by the vectors
$$V^{i}(x,t,s):=\Bigl( \frac{\partial f_{i}}{\partial x_{1}}(x,t),\dots,\frac{\partial f_{i}}{\partial x_{n}}(x,t), 0,  \dots, 0, Ls_{i}^{L-1},0,\dots,0\Bigr)\in  (\bR^{n}\times \bR \times \bR^{m})^{*}, $$
for $i=1,\ldots , m$.

We consider paths $\gamma(\lambda) :=(x_{\lambda},t_{\lambda},s_{\lambda})\in  \bR^{n}\times \bR \times \bR^{m}$ depending on a parameter $\lambda\in \bR$,
such that  $(x_{\lambda},t_{\lambda},s_{\lambda})\rightarrow(x_{0},0,0)$  when $\lambda \to 0$.

We claim that we obtain the cotangent space $T^{*}$ dual to $T$ by first considering the limits, along all such paths $\gamma(\lambda)$, of all the linear combinations
$\sum_{i=1}^{m} \alpha_{\lambda}^{i} V^{i}(x_{\lambda},t_{\lambda},s_{\lambda})$ viewed as elements in the projective space $\bP(\bR^{n}\times \bR \times \bR^{m})^{*}$, with coefficients $\alpha_{\lambda}^{i}$ depending of the parameter $\lambda$ too, and such that $\| \alpha_{\lambda}\| =1$ for any $\lambda\not= 0$. Finally the full cotangent space $T^{*}$ is the set of all scalar multiples of these limits.

Let us divide $\sum_{i=1}^{m} \alpha_{\lambda}^{i} V^{i}(x_{\lambda},t_{\lambda},s_{\lambda})$ by the positive real $\nu_{F_{t}}(x)$, and compute the limits. We find:
$$  \lim_{\lambda\to 0}\frac{\| \sum_{i=1}^{m} \alpha_{\lambda}^{i}\grad_{x} f_{i}(x_{\lambda},t_{\lambda})\| }{\nu_{F_{t_{\lambda}}}(x_{\lambda})} \ge 1 \ \mbox{ and } \ 
\lim_{\lambda\to 0}\frac{\| \sum_{i=1}^{m} \alpha_{\lambda}^{i} L s_{\lambda, i}^{L-1}\|}{\nu_{F_{t}}(x_{\lambda})} =0
$$
where the former follows by the definition of $\nu_{F_{t}}(x)$ as a minimum. 
The later limit justifies as follows: firstly, by 
the definition of $X$,  we have $\|s_{i}^{L-1}\| = \|f_{i}(x,t)\|^{\frac{L-1}{L}} \le  \|F(x,t)\|^{\theta +\e}$, where  $\frac{L-1}{L}= \theta + \e$ and therefore $\e >0$  by our choice of $L$. Next,  by applying the hypothesis \eqref{eq:cond1} we get : 
$$\|F(x,t)\|^{\theta +\e} \le c(x_{0}) \nu_{F_{t}}(x)  \|F(x,t)\|^{\e},$$ 
where $\|F(x,t)\|^{\e}$ converges to 0 when $x\to x_{0}$ and $t\to 0$.

This shows that the limit:
\begin{equation}\label{eq:limcotang}
   \lim_{\lambda\to 0}\frac{1}{\nu_{F_{t}}(x)} \sum_{i=1}^{m} \alpha_{\lambda}^{i} V^{i}(x_{\lambda},t_{\lambda},s_{\lambda})
\end{equation}
represents a nontrivial direction in $\bP(\bR^{n}\times \bR \times \bR^{m})^{*}$ with its last $m$ positions equal to $0$, and this holds for any path $\gamma(\lambda)$ as considered above.

Since the projective cotangent space $\bP T^{*}$ is the set of all directions of type \ref{eq:limcotang}, it follows that 
$T^{*}$ (which has dimension $m+1$) is contained in $N\times \bR^{*} \times \{0\}_{m}$, where $N\subset (\bR^{n})^{*}$ is some linear subspace. Therefore $T$ equals 
 $A\times \{0\}_{1} \times \bR^{m}$, for some linear subspace $A\subset \bR^{n}$, and since  $\dim T =n$, we get $\dim A = n-m$.
\end{proof}

We continue the proof of the theorem. Let us consider the projection map $p:(X,0)\longrightarrow(\bR \times \bR^{m},0)$, $(x,t,s)\mapsto(t,s)$. Since $(t, s)$ is a regular value of $p$, we have the following inclusion
\begin{equation}\label{eq:inclu14}
T=\lim_{k\rightarrow\infty}T_{(x,t,s)}\bigl(\pi^{-1}(t)\cap X\bigr)\supset \lim_{k\rightarrow\infty}T_{(x,t,s)}\bigl(p^{-1}(t,s)\cap X\bigr) =: T^{'}
\end{equation}
where  $\dim T' = \dim T_{(x,t,s)}\bigl(p^{-1}(t,s)\cap X\bigr)=n-m$,
and  $p^{-1}(t,s)\cap X=\widetilde F^{-1}\bigl(\widetilde F(x, t)\bigr)$.

  In order to complete the proof of the theorem, we need  to show the inclusion $T_{(x_{0},0,0)}W\subset T^{'}$.
  Both spaces, $T'$ and $T_{(x_{0},0,0)}W$, are included in $\bR^{n}\times \{0\}_{1}\times \{0\}_{m}$ by their definition, and they are also included in $T$, by \eqref{eq:inclu13} and \eqref{eq:inclu14}. Since we have the identification  $T= A \times \{0\}_{1} \times \bR^{m}$ by the above Lemma \ref{l:keylemma},  it follows that those two spaces must be included in the intersection, which is $A \times \{0\}_{1} \times \{0\}_{m}$.
 
 Now, since the equality of dimensions $\dim T' = \dim A = n-m$, this inclusion must yield an equality: $T'=A \times \{0\}_{1} \times \{0\}_{m}$. 
 This implies the inclusion $T_{(x_{0},0,0)}W \subset \bR^{n}\times \{0\}_{1}\times \{0\}_{m} \cap T$, thus we get 
 $T_{(x_{0},0,0)}W \subset T'$, which finishes the proof of our theorem.
 
\end{proof}

\subsection{Parusi\' nski type inequality in case of map germs}

We still consider analytic map germs  $F:(\bK^n\times\bK, 0)\to(\bK^m, 0)$, $n>m\ge 1$, regarded as a one-parameter deformations of $F_{0}:=F(x,0)$, and $F_{t}(x) = F(x,t)$. For simplicity, we still assume that $\bK = \bR$.
The following result extends \cite{JST} where one used a week Parusi\' nski type inequality \cite{Pa}.
%%%%%%%

%%%%%%%%%%%%%%%%
\begin{theorem}\label{t:tame}   
Let $F_{t}(x) =F(x,t)$ be an analytic deformation of $F_{0}$ such that the map germ $F:(\bR^n\times\bR, 0)\to(\bR^m, 0)$ is $\partial$-Thom regular. If  $F$ satisfies the condition
\begin{equation}\label{eq:cond}
 \left\{  \begin{array}{ll}  \mbox{ For } \mbox{ any }  y\in F_{0}^{-1}(0)\cap \Sing F_{0} \m \{0\}   \mbox{ there is } \\
  \mbox{  a constant } c(y) >0   \mbox{ for which the following inequality holds: }\\
     \bigl\| \frac{\partial F}{\partial t}(x,t) \bigr\| \le c(y) \nu_{F_{t}}(x)    \  \mbox{ when } (x,t) \to (y,0), \ (x,t) \notin \Sing \widetilde F
     \end{array} \right.
\end{equation}
 then the associated map germ $\widetilde F$ is $\partial$-Thom regular.
 \end{theorem}
\begin{proof}
According to Definition \ref{def:stratified thom}, what we have to show amounts to proving the existence of a Whitney (a)-regular stratification of $\widetilde F^{-1}(0,0) \cap \Sing \widetilde F \m \{(0,0)\}$ in some ball $B'$ such that it is Thom (a$_{\widetilde F}$)-regular with respect to the stratum $B'\m \widetilde F^{-1}\bigl(\widetilde F(B'\cap \Sing \widetilde F)\bigr)$.  Let us first observe the equalities of set germs at the origin: 

\begin{lemma}\label{eq:coherent}
\[
\widetilde F^{-1}(0,0) \cap \Sing \widetilde F \m \{0\} = F_{0}^{-1}(0)\cap \Sing F_{0} \m \{0\} = F^{-1}(0)\cap \Sing F \cap \{t=0\} \m \{0\}.
\]
\end{lemma}
\begin{proof}
The first equality, as well as the inclusion ``$\supset$'' in place of the second equality are direct consequences of the respective definitions of the singular sets. 

To show the reciprocal inclusion  ``$\subset$'' in place of the second equality, we will use the condition \eqref{eq:cond} at the point $(y,0)\in F_{0}^{-1}(0)\cap \Sing F_{0}$, where $y\not= 0$.  

We still use the notation $F(x,t)=(f_{1}(x,t),\dots,f_{m}(x,t))$.
Firstly, we observe that there is some linear combination $\sum_{i=1}^{m} a_i \gradx f_i (y,0)$ equal to 0 since the gradients  are linearly dependent  at the singular point $(y,0) \in F_{0}^{-1}(0)\cap \Sing F_{0} \m \{0\}$, and thus Definition \eqref{eq:niu} tells that $\nu_{F_{0}}(y) =0$.

Secondly, using the same fixed coefficients $a_{i}$,  we have 
$$\lim_{t\to 0, x\to y} \sum_{i=1}^{m} a_i \gradx f_i (x,t) = \sum_{i=1}^{m} a_i \gradx f_i (y,0),$$
 by the continuity of the gradient functions. 
 By the Definition \eqref{eq:niu} as a ``minimum'' we get $\nu_{F_{t}}(x) \le \| \sum_{i=1}^{m} a_i \gradx f_i (x,t) \|$, and since the right hand side is equal to 0, it follows that $\lim_{t\to 0, x\to y} \nu_{F_{t}}(x) =0$.
 It is at this moment that we apply the condition \eqref{eq:cond} to deduce, by continuity,  that $\frac{\partial F}{\partial t}(y,0) =0$. This shows that the last column in the Jacobian matrix of $F$ at $(y,0)$ is zero, thus we get  $(y,0)\in F^{-1}(0)\cap \Sing F \cap \{t=0\} \m \{0\}$, which ends the proof of our lemma. 
\end{proof}

By our hypothesis, there is a semi-analytic Whitney (a)-regular stratification $\cS$ of  a ball $B\subset \bR^{n}\times \bR$ which is Thom (a$_{F}$)-regular, and such that $F^{-1}(0)\cap \Sing F$ is a union of strata. Let  $(y,0)\in B$,  $y\not= 0$,  be a point on some stratum $V\in \cS$, where $V\subset F^{-1}(0)\cap \Sing F$.

Let $(x_{t},t)\to (y,0)$ be a continuous path  such that $(x_{t},t)\not\in \Sing F$. Let  $T_{(x_{t},t)}F^{-1}(s_{t})$ denote the tangent space at some smooth point $(x_{t},t)$ of the fibre over  $s_{t}:=F(x_{t},t)$. 
 The assumed $\partial$-Thom (a$_{F}$)-regularity condition at $(y,0)$ amounts to the following property: 
 for any choice of the path $(x_{t},t)\to (y,0)$ as above, we have the inclusion:
 
\begin{equation}\label{eq:incl1}
  T := \lim_{(x_{t},t)\to (y,0)}T_{(x_{t},t)}F^{-1}(s_{t}) \supset T_{(y,0)}V,
\end{equation}
 whenever the limit exists in the appropriate Grassmannian, in which case we have $\dim T = n-m+1$.
  
We  next consider the slice of the stratification $\cS$ by $\{t=0\}$, consisting of the sets $V':=V\cap \{t=0\}$ for all $V\in \cS$. There exists the roughest semi-analytic Whitney (a)-regular stratification $\cS'$ of the central fibre $\widetilde F^{-1}(0,0) = F_{0}^{-1}(0)$ which refines this slice stratification, in particular the sets $V'$ are unions of strata of $\cS'$.

\begin{lemma}\label{l:A}
Under the hypotheses of Theorem \ref{t:tame}, one has $T \not\subset \bR^{n} \times \{0\}$, equivalently: $T$ is transversal to $\bR^{n} \times \{0\}$ at $(y,0)$
\end{lemma}
\begin{proof}
 The normal space to the tangent space $T_{(x,t)}F^{-1}(s_{t})$ is spanned by
the gradient vectors $\grad f_{i}(x,t)$, for $i=\overline{1, m}$.   We consider paths $\gamma(t) :=(x_{t},t)\in  B\subset \bR^{n}\times \bR$
depending on the parameter $t$ in some small neighbourhood of the origin $0\in \bR$. 
As shown in the proof of Lemma \ref{l:keylemma}, the limit cotangent space $T^{*}$  is the set of the scalar multiples of all the limits when $t\to 0$ of the linear combinations
$\sum_{i=1}^{m} \alpha_{\lambda}^{i} \grad f_{i}(x_{t},t)$ with $\| \alpha_{\lambda}\| =1$.

By absurd, let us suppose that $T\subset \bR^{n} \times \{0\}$. This is equivalent to $T^{*}\ni (0, \ldots, 0,1)$. 
Using the inequality \eqref{eq:cond},  we find that, for any path $(x_{t},t)$, the $(n+1)$-dimensional vector:
$$\lim_{t\to 0} \frac{1}{\nu_{F_{t}}(x_{t})} \sum_{i=1}^{m} \alpha_{\lambda}^{i} \grad f_{i}(x_{t},t)$$
has on the first $n$ entries the coordinates a vector of modulus bounded from below by 1, whereas the last entry is bounded from above by a positive constant. This tells that no such vector can be of the form $(0, \ldots , 0, \lambda)$, with $\lambda \not= 0$, thus this cannot be a vector in $T^{*}$, hence we get a contradiction.
\end{proof}
%%%%%%%%%%%%%%
In order to show that  $\cS'$ is a $\partial$-Thom stratification for the map $\widetilde F$  (cf Definition \ref{d:partialThom}) we need to prove the $\partial$-Thom (a$_{\widetilde F}$)-regularity condition at some point $(y,0)$. This amounts to showing that
for any choice of a sequence $(x_{t},t)\to (y,0)$ such that $(x_{t},t)\not\in \Sing F$ and $s_{t}:= F_{t}(x_{t})$, we must have the inclusion (where again we assume without loss of generality that the limit exists in the appropriate Grassmannian):
\begin{equation}\label{eq:incl2}
 T' :=  \lim_{(x_{t},t)\to (y,0)}T_{(x_{t},t)}(F_{t}^{-1}(s_{t})\times\{t\}) \supset T_{(y,0)}V',
\end{equation}
where $V'$ is the stratum of $\cS'$ which contains $(y,0)$, in particular one has $T_{(y,0)}V\supset  T_{(y,0)}V'$.

Both sides of \eqref{eq:incl2} are included in $\bR^{n}\times \{0\}$ by construction, and also in $T$, thus they are included in 
the intersection $T \cap \bR^{n}\times \{0\}$, which by Lemma \ref{l:A} is a transversal intersection, and hence of dimension $n-m$.
But since $T' \subset  T   \cap  \bR^{n}\times \{0\}$ and $\dim T'=n-m$, it follows that we have equality: $T' =  T \cap \bR^{n}\times \{0\}$. Consequently $T_{(y,0)}V' \subset T'$, and this ends our proof.
\end{proof}

\begin{remark}\label{ex:4}
One may be tempted to say that Theorem \ref{t:tame} implies Theorem  \ref{t:deformation}.
This turns out to be true in case of function germs, see Remark \ref{r:AG} below,  but it is not true in the setting of map germs, as we will show in the following.

Let us consider the map germ $F(x,y,z,t)= (x,g(x,y,z,t))=(x, (y(x^2 + y^2 + z^2))^3 + x^2(y(x^2 + y^2 + z^2))+tx^{k})$ with $k\geq 5$, as deformation of $F_{0} :=  (x,g(x,y,z,0))$. Then $\Sing F=V(F)=\{x=y=0\}$. By considering paths $\phi(s)=(s,0,z_{0},0)$ such that $\lim_{s\to 0}\phi(s)=(0,0,z_{0},0)\in\Sing F_{0}\m\{0\}$, we have $\ord_{s}\nu_{F_{t}}(\phi(s))=2$ and $\ord_{s}x(s)=1$, and therefore $F$ does not satisfy the  inequality \eqref{eq:cond1}. 
	
	On the other hand, for any analytic path $\phi(s)=(x(s),y(s),z(s),t(s))$ such that  $\underset{s\rightarrow0}{\lim}\phi(s)=(0,0,z_{0},0)\in\Sing F_{0}\m\{0\}$, we have $\ord_{s}\nu_{F_{t}(\phi(s))}\leq 2\cdot \ord_{s}x(s)$. Since $k\geq 5$, we get:
	$\ord_{s}\bigl\|\frac{\partial F}{\partial t}\bigr\|=k \cdot \ord_{s}x(s)>\ord_{s}\nu_{F_{t}}(\phi(s)),$
	which implies that condition \eqref{eq:cond} holds, and therefore, by Theorem \ref{t:tame}, $\widetilde F=(F, t)$ is $\partial$-Thom regular. 
\end{remark}

In case of  deformations of function germs, the above theorem extends the setting of function germ deformations:
%%%%%%%%%%%%%%%%
\begin{corollary}\cite{JST} \label{c:tame}
Let $F_{t}(x) = F(x,t)$ be a C$^{1}$-family of analytic function germs  $F_{t}: (\bR^{n}, 0) \to (\bR, 0)$
 which satisfies the condition:

 \begin{equation}\label{eq:cond2}
 \left\{  \begin{array}{ll}  \mbox{ For } \mbox{ any }  y\in \Sing F_{0} \m \{0\}, \mbox{ there is } c(y) >0 \mbox{ such that } \mbox{ for } (x,t)\not\in \Sing \widetilde F: \\
    \bigl| \frac{\partial F}{\partial t}(x,t) \bigr| \le  c(y) \biggl\| \frac{\partial F}{\partial x_{1}}(x,t) , \ldots , \frac{\partial F}{\partial x_{n}}(x,t)   \biggr\|   \  \mbox{ when } (x,t) \to (y,0).
     \end{array} \right.
\end{equation} 
  
 Then the deformation $\widetilde F$ is $\partial$-Thom regular.
 \fin
 \end{corollary}

%\red{General question : Can $c(y)$ tend to infinity as $y$ converges to 0? Examples?}

\begin{remark}\label{r:AG}
 By applying the classical \L ojasiewicz inequality to the function germ $F(x,t)$ over some neighbourhood $U$ of $0\in \bR^{n}\times \bR$, we get the existence of $0<\theta<1$ and $C>0$ such that the inequality:
\begin{equation}\label{eq:condfunction}
 \|F(x,t)\|^{\theta}\leq C \biggl\|\frac{\partial F}{\partial t}(x,t) , \frac{\partial F}{\partial x_{1}}(x,t) , \ldots , \frac{\partial F}{\partial x_{n}} (x,t)  \biggr\|
\end{equation}
holds over $U$.  From \eqref{r:AG} and \eqref{eq:cond2} we then deduce:
\begin{equation}\label{eq:condfunct3}
 \left\{  \begin{array}{ll}  \mbox{ For } \mbox{ any }  x_{0}\in \Sing F_{0} \m \{0\}, \mbox{ there is } k(x_{0}) >0 \mbox{ such that }:  \\
  \|F(x,t)\|^{\theta}\leq k(x_{0}) \biggl\|  \frac{\partial F}{\partial x_{1}}(x,t), \ldots , \frac{\partial F}{\partial x_{n}}(x,t) \biggr\|
  \  \mbox{ when } (x,t) \to (x_{0},0).
     \end{array} \right.
\end{equation} 
  While this shows the implication \eqref{eq:cond2} $\Rightarrow$ \eqref{eq:condfunct3}, let us observe that the stronger condition \eqref{eq:cond2} has the advantage  upon  \eqref{eq:condfunct3} that it is easier to test.
Remark that the converse implication is not true, see Example \ref{ex:3}.  For map germs, even the direct implication is not true, as seen in Remark \ref{ex:4}.

In \cite{AG} one considers the condition \eqref{eq:condfunct3} for deformations of function germs $F$ under the restriction that the associated map $\widetilde F$  has an isolated critical value at $0$.  Let us therefore point out that our Theorem \ref{t:deformation} represents a far-reaching extension of the main result \cite[Theorem 3.5]{AG} since it works for deformations of map germs $F$, and without any special assumption on the singular locus of $\widetilde F$. 
\end{remark}

\begin{example}\label{ex:3}
Let $F: (\bR^{3}\times \bR,0) \to (\bR, 0)$ be the deformation $F(x,y,z,t) :=(x^{4}+y^{2}z^{2})(x-t)+t^{2}x$, of parameter $t\in \bR$.
When $t=0$, we have $F_{0}(x,y,z)=x(x^{4}+y^{2}z^{2})$ and $\Sing F_{0}(x,y,z)=\{x=y=0\}\cup\{x=z=0\}$. For the paths $\phi(s)=(x(s),y(s),z(s),t(s))=\bigl(s,s^{4},z_{0},ks^{3}\bigr)$ with $z_{0}\neq0$ and $k \gg 1$, making $s\to 0$ we see that there is  no $c(z_{0})>0$ such that condition \eqref{eq:cond2} holds.  
Nevertheless, it appears that the condition \eqref{eq:condfunct3} holds for some $\frac{4}{5}<\theta<1$. Thus  \eqref{eq:condfunct3} does not imply \eqref{eq:cond2}, however  $\widetilde F$ is $\partial$-Thom by Corollary \ref{c:tame}.
\end{example}

%%%%%%%%%%%%%
%%%%%%%%%%%%%%%%
  %%%%%%%%%%%%%%%%%
 \section{Conservation of the $\partial$-Thom regularity in compositions of deformations} \label{s:ag}

The existence of the Milnor-Hamm tube fibration for the composition of map germs is treated in the recent paper \cite{CJT} with criteria involving the $\rho$-regularity condition.  In case of the composition of deformations of function germs, in the aim of insuring the existence of the Milnor-Hamm tube fibration, the authors of  \cite{AG} consider the property \eqref{eq:condfunct3}  of Remark \ref {r:AG}
within the class of deformations of function germs $Q$ such that $\widetilde Q$ has isolated singular value.
They leave open the question whether or not this condition is preserved by the composition of such deformation maps. 

  As  the $\partial$-Thom regularity is a sufficient condition which insures the existence of the Milnor-Hamm tube fibration for the composition of deformations of function germs, our answer to the dilemma is as follows: instead of the property \eqref{eq:condfunct3},  which might not be preserved by compositions, one should consider the property ``\emph{$\widetilde F$ is $\partial$-Thom regular}''.

 Our next result shows that the property ``\emph{$\partial$-Thom regularity}''  of $\widetilde G$ and $\widetilde F$ is indeed preserved by the composition $\widetilde H = \widetilde G\circ \widetilde F$. Moreover, unlike \cite{AG}, we do not need to assume 
 that $\widetilde G$ has an isolated critical value.  Such a result turns out to hold under more general conditions for compositions of map germs. 
 
 %%%%%%%%%%%%
\begin{theorem}\label{t:Thomcompo}
   If  $\widetilde F$ and $\widetilde G$  are $\partial$-Thom regular, and if $\widetilde F$ has an isolated singular value, then the composed map $\widetilde H = \widetilde G\circ \widetilde F$ is  $\partial$-Thom regular.
\end{theorem}
%%%%%%%%%%%%%%%%
\begin{proof}
 The proof reduces to the case $\Sing \widetilde H \cap \widetilde H^{-1}(0)$ has positive dimension as a set germ at $0\in \bR^{n}\times \bR$, since otherwise $\widetilde H$ is $\partial$-Thom regular by definition.
 
Since $\widetilde F$ is $\partial$-Thom regular, there exists a  stratification $\cW$ of the open ball $B_{\e}\subset \bR^{n}\times \bR$, for some  $\epsilon>0$ small enough, which verifies the conditions of  Definition \ref{def:stratified thom}.

Consider a sequence $\{x_{i}\}_{i\in \bN}\subset B_{\varepsilon}\setminus \widetilde H^{-1}(\Disc(\widetilde H))$ such that $x_{i}\rightarrow y\in V(\widetilde H)$.

Let first $y \in V(\widetilde F)$. We denote by 
$W\in\cW$ the stratum such that $y\in W$. Remark that $x_{i}$ are regular points of $\widetilde F$
by the assumption that $\widetilde F$ has an isolated singular value. 

Assuming (without loss of generality) that the limits exist in the appropriate Grassmannians, and using the Thom (a$_{\widetilde F}$)-regularity, we get the first inclusion in:
$$T_{y} W \subset\lim_{i\rightarrow\infty}T_{x_{i}}\widetilde F^{-1}(\widetilde F(x_{i})) \subset  \lim_{i\rightarrow\infty}T_{x_{i}}\widetilde H^{-1}(\widetilde H(x_{i})) ,$$
whereas the second inclusion is due to the inclusion of nonsingular fibres 
 $\widetilde F^{-1}(\widetilde F(x_{i}))\subset \widetilde H^{-1}(\widetilde H(x_{i}))$, for any $i$.
 This shows that the pair $\bigl(B^{n}_{\varepsilon}\setminus \widetilde H^{-1}(\Disc(\widetilde H)\bigr), W)$ satisfies Thom (a$_{\widetilde H}$)-condition.
 
 If now $y\in V(\widetilde H)\setminus V(\widetilde F)$, then $y$ is a regular point of $\widetilde F$. Since  $\widetilde H(x_{i})\notin \Disc(\widetilde H)$, we also have $\widetilde H(x_{i})\notin\Disc(\widetilde G)$. 
	
Since $\widetilde G$ is $\partial$-Thom regular by our hypothesis,  we may consider an open ball $B_{\delta}\subset \bR^{p}\times \bR$ together with 
a stratification $\cS$ which satisfy Definition  \ref{def:stratified thom} for $\widetilde G$ instead of $\widetilde F$. Therefore the point $\widetilde F(y)\subset V(\widetilde G)$ belongs to a positive dimensional stratum $S\subset \cS$. The assumed
(a$_{\widetilde G}$)-regularity of the pair $\bigl(B_{\delta}\setminus \widetilde G^{-1}(\Disc(\widetilde G)), S\bigr)$ yields:
\begin{equation}\label{eq:Thom}
		T_{\widetilde F(y)}S  \subset \lim_{i\rightarrow\infty} T_{\widetilde F(x_{i})} \widetilde G^{-1}(\widetilde H(x_{i})),
\end{equation}
where we may again tacitly assume that the limit exist in the appropriate Grassmannian.

The map $\widetilde F$  is a submersion in some neighbourhood of $y$. By applying to \eqref{eq:Thom} the inverse map $(T_{*}\widetilde F)^{-1}$, which commutes with the limit ``$\lim_{i\rightarrow\infty}$'',  we get the equalities
$T_{y}\widetilde F^{-1}(S) = (T_{y}\widetilde F)^{-1}(T_{\widetilde F(y)}S)$ and
 $$(T_{x_{i}}\widetilde F)^{-1}(\lim_{i\rightarrow\infty} T_{\widetilde F(x_{i})} \widetilde G^{-1}(\widetilde H(x_{i}))) =
		\lim_{i\rightarrow\infty} (T_{x_{i}}\widetilde F)^{-1}(\widetilde G^{-1}(\widetilde H(x_{i}))),$$
therefore we obtain the inclusion:
\begin{equation}\label{eq:Inclusion}
	T_{y}\widetilde F^{-1}(S) \subset 
		\lim_{i\rightarrow\infty} (T_{x_{i}}\widetilde F)^{-1}(\widetilde G^{-1}(\widetilde H(x_{i}))) = \lim_{i\rightarrow\infty} (T_{x_{i}}\widetilde H)^{-1}(\widetilde H(x_{i})).
\end{equation}
This  shows that the pair $\bigl(B_{\varepsilon}\setminus \widetilde H^{-1}(\Disc(\widetilde H)), \widetilde F^{-1}(S)\bigr)$ is Thom (a$_{\widetilde H}$)-regular, and this  holds for all  strata $S\in \cS$ which are inside $V(\widetilde G)\m \{0\}$.

We have now a stratification of $V(\widetilde H)$ which consists of the newly defined strata $\widetilde F^{-1}(S)$ for all $S\in \cS$, $S\subset V(\widetilde G)\m \{0\}$, and all the strata $W\in \cW'$, $W\subset V(\widetilde F)$. What we have proved up to now is that
the pair consisting of $B_{\varepsilon}\setminus \widetilde H^{-1}(\Disc(\widetilde H))$ and any of the above enumerated strata in $V(\widetilde H)$ satisfies the Thom (a$_{\widetilde H}$)-regularity condition.

There is still one more thing to check:  the Whitney (a)-regularity of the appropriate couples of strata.
Since $\widetilde F$ is a submersion on $V(\widetilde H)\m V(\widetilde F)$, the inverse image by $\widetilde F$ preserves the Whitney (a)-regularity of every pair $(S', S'')$ of strata of $\cS$ which are inside $V(\widetilde G)\m \{0\}$, and such that $S''\subset \overline{S'}$.  But since the Whitney (a)-regularity may not be satisfied with respect to the strata of $\cW'$ inside $V(\widetilde F)$,
we need to refine the stratification $\cW'$ into a Whitney (a)-regular stratification $\cW''$ such that the pairs $(S,W)$ are Whitney (a)-regular, for any $S\in \cS$ and $W\in \cW''$, and such that $W\subset \overline{S}$.
   
Then the germ at $0\in \bR^{n}\times \bR$ of the stratification $\cQ := \cS \bigsqcup \cW''$ 
 is a  $\partial$-Thom stratification for $\widetilde H$.
This ends our proof.

\end{proof}

\begin{example}\label{ex:Thomcompo}(after \cite{AG})
Consider the maps germs $F : (\bR^3, 0) \to (\bR, 0)$, where $F(y, z, t) = y(y^2 + z^2+t^{2})$, and $G : (\bR^2, 0) \to (\bR, 0)$, where $G(y, t) = y^3 + t^2y$. It turns out that  $F$ satisfies the condition \eqref{eq:condfunct3} in a neighbourhood of the origin, whereas $G$ has a single singularity at the origin thus satisfies \eqref{eq:condfunct3} in a trivial manner.  (One can also check easily that condition \eqref{eq:cond2} is  satisfied too.) 

Then both $\widetilde F$ and $\widetilde G$ are $\partial$-Thom regular, by Theorem \ref{t:deformation}.  Since $\widetilde F$ has an isolated critical value, we may now directly apply our Theorem \ref{t:Thomcompo} to conclude that $\widetilde H=\widetilde G\circ \widetilde F$ is  $\partial$-Thom regular, and therefore $\widetilde H$ has a Milnor-Hamm fibration by Corollary \ref{c:milnorfib}. We observe that Theorem \ref{t:Thomcompo} provides a shortcut and we do not need to verify again the condition \eqref{eq:condfunct3} for the composition $H =  G\circ F$ as done in \cite[Example 3.7]{AG} in order to obtain a Milnor-Hamm fibration. 
\end{example}

 %And as we have remarked before, our test Theorem \ref{t:Thom} for $\partial$-Thom regularity of the ingredients $%\widetilde F$ and $\widetilde G$ (Theorem \ref{t:Thom}) is more general than that of \cite{AG}.

 %%%%%%

%%%%%%%%%%%%%%%%%%%%%%%%%

%%%%%%%%%%

%%%%%%%%%%%%%%%%%%%%%%%%%%%%%%%%%%%%%%%%%%%%%

\end{document}